\documentclass[10pt]{article}
\usepackage[utf8]{inputenc}
\usepackage[english]{babel}
\usepackage{amsmath}
\usepackage{amsthm}
\usepackage{amsfonts}
\usepackage{amssymb}
\usepackage{tikz}
\usepackage{makeidx}
\usepackage{lmodern}
\newtheorem{theorem}{Theorem}[section]
\newtheorem{conjecture}{Conjecture}[section]
\newtheorem{definition}{Definition}[section]
\newtheorem{example}{Example}[section]
\newtheorem{lemma}{Lemma}[section]
\newtheorem{corollary}{Corollary}

\usepackage{graphicx}
\usepackage{longtable}
\graphicspath{ {./Desktop/} }
\font\smallit=cmti10

\begin{document}


\begin{center}
{\bf The $2$-adic valuation of the general degree 2 polynomial in 2 variables}
\vskip 20pt
{\bf Shubham}\\
{\smallit School of Physical Sciences,  Jawaharlal Nehru University,  New Delhi, 110067,  India}\\
{\tt shubham01nitw@gmail.com}
\end{center}
\vskip 30pt
\vskip 30pt

\centerline{\bf Abstract}

\noindent
  The $p$-adic \textit{valuation} of a polynomial can be given by its \textit{valuation tree}. This work describes the $2$-adic \textit{valuation tree} of the general degree 2 polynomial in 2 variables.
\pagestyle{myheadings}

\thispagestyle{empty} 
\baselineskip=15pt 
\vskip 30pt
\section{Introduction}
For $n \in \mathbb{N}$, the highest power of a prime $p$ that divides $n$ is called the $p$-adic valuation of the $n$.  This is denoted by $v_p(n)$.  Legendre establishes the following result about $p$-adic valuation of $n!$ in \cite{one} \begin{center}
     $v_p(n!) = \sum_{k=1}^{\infty} \lfloor \frac{n}{p^k}  \rfloor = \frac{n-s_p(n)}{p-1}$;
\end{center}where $s_p(n)$ is the sum of digits of $n$ in base $p$.  It is observed in \cite{two} that $2$-adic valuation of central binomial coefficient is   $s_2(n)$ i.e. \begin{center}
$v_2(C_n)$ = $s_2(n)$ where $C_n$ = $\binom{2n}{n}$.
\end{center}It follows from here that $C_n$ is always an even number and $C_n/2$ is odd when $n$ is a power of 2.  It is called a closed form in [2]. The definition of closed form depends on the context.  This has been discussed in[3, 4]. 
\par The work presented in \cite{two} forms part of a general project initiated by Victor H.  Moll et al to analyse the set \begin{center}
 $V_x = \{ v_p(x_n):n \in \mathbb{N}\}$;
\end{center}for a given sequence $x =\{x_n\}$.  \\ The $2$-adic valuation of $n^2-a$ is studied in \cite{two}. It is shown that $n^2-a, a\in \mathbb{Z}$ has a simple closed form when $a \not \equiv 4,7 \mod 8$.  For these two remaining cases the valuation is quite complicated.  It is studied by the notion of the \textit{valuation tree}. 

Given a polynomial $f(x)$ with integer coefficients, the sequence $\{v_2(f(n)): n \in \mathbb{N}\}$ is described by a tree. This is called the valuation tree attached to the polynomial $f$.  The vertices of this tree corresponds to some selected classes \begin{center}
$C_{m,j} = \{2^mi+j: i \in \mathbb{N}\}$,
\end{center}starting with the root node $C_{0,0} = \mathbb{N}$.  The procedure to select classes is explained below in the example (1.1).  Some notation for vertices of tree are introduced as follows: \begin{definition} A residue class $C_{m,j}$ is called \textit{terminal} if $v_2(f(2^mi+j))$ is independent of $i$.  Otherwise it is called \textit{non-terminal}.  The same terminology is given to vertices corresponding to the class $C_{m,j}$.  In the tree,  terminal vertices are labelled by their constant valuation and non-terminal vertices are labelled by a $*$. \end{definition}
\begin{example} Construction of valuation tree of $n^2 + 5$ is as follows:
note that $v_2(1+5)$ = 1 and $v_2(2+5$) is 0. So node $v_0$ is non terminating.  Hence it splits into two vertices and forms the first level.  These vertices correspond to the residue classes $C_{1,0}$ and $C_{1,1}$.  We can check that both these nodes are terminating with valuation 0 and 1. So the valuation tree of $n^2+5$ is given as follows:\end{example}
    
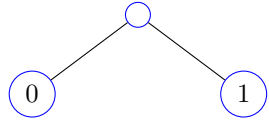
\hspace*{3cm}   \begin{tikzpicture}
      [sibling distance=8em,level distance=3em,
      every node/.style={shape=circle,draw= blue,align=right}]
        \node{}
        child{node{0}
        } 
        child{node{1}}
              ;
        
      \end{tikzpicture} \par The main theorem of \cite{two} is as follows:\begin{theorem}
      Let $v$ be a non-terminating node at the $k$-th level for the valuation tree of $n^2 + 7$. Then $v$ splits into two vertices at the $(k + 1)$-level.  Exactly one of them terminates, with valuation $k$. The second one has valuation at least $k + 1$.
      \end{theorem}
     \par The $2$-adic valuation of the Stirling numbers is discussed in \cite{five}. The numbers $S(n,k)$ are the number of ways to partition a set of $n$ elements into exactly $k$ non-empty subsets where $n \in \mathbb{N}$ and $0\leq k \leq n$. These are explicitly given by \begin{center}
      $S(n,k) = \frac{1}{k!}\sum_{i=0}^{i=k-1}(-1)^{i}\binom{k}{i}(k-i)^{n}$;
\end{center}or, by the recurrence\begin{center}
$S(n,k) = S(n-1,k-1) + kS(n-1,k)$;
\end{center}with initial condition $S(0,0) = 1$ and $S(n,0) = 0 $ for $n>0$.  \\The $2$-adic valuation of $S(n,k)$ can be easily determined for $1\leq k \leq 4$ and closed form expression is given as follows:\begin{center}
$v_2(S(n,1))$ =0 = $v_2(S(n,2))$;\\
\end{center}
      \begin{center}
        \[
   v_2(S(n,3))= 
\begin{cases}
   1,& \text{if } n \hspace*{.2cm}  \text{is}\hspace*{.2cm} \text{even}\\
    0,              & \text{otherwise}
\end{cases}
\]
      \end{center}
      \begin{center}
        \[
   v_2(S(n,4))= 
\begin{cases}
   1,& \text{if } n \hspace*{.2cm} \text{is}\hspace*{.2cm} \text{odd}.\\
    0,              & \text{otherwise}
\end{cases}
\]
      \end{center}
      The important conjecture described there is that the partitions of $\mathbb{N}$ in classes of the form \begin{center} $C_{m,j}^{(k)} = \{2^mi+j:$ $i \in \mathbb{N}$ and starts at the point where $2^mi+j \geq k$ \} \end{center} leads to a clear pattern for $v_2(S(n,k))$ for $k$ $\in \mathbb{N}$ is fixed. We recall that the parameter $m$ is called the level of the class.   The main conjecture of \cite{five} is now stated:
      \begin{conjecture}Let $k \in\mathbb{N}$ be fixed.  Then we conjecture that \begin{itemize}
 \item[(a)] 
there exists a level $m_0(k)$ and an integer $\mu(k)$ such that for any $m\geq m_0(k)$,  the number of non-terminal classes of level $m$ is $\mu(k)$, independently of $m$;
\item[(b)] moreover, for each $m \geq m_0(k)$, each of the $\mu(k)$ non-terminal classes splits into one terminal and one non-terminal subclass. The latter generates the next level set.
   \end{itemize}
      \end{conjecture}This conjecture is only established for the case $k = 5$.
      A similar conjecture is given in \cite{six} for the $p$-adic valuation of the Stirling numbers.
     \par In this work,  we discuss the set \begin{center}
      $V_f$ = $\{ v_p(f(m,n)): m,n \in \mathbb{N}\}$;
\end{center}where $f(X,Y) \in \mathbb{Z}[X,Y]$, by the generalized notion of the \textit{valuation tree}.  We believe that the $p$-adic valuation of two variable polynomials has not been studied before.  We define the $p$-adic valuation tree as follows:         
 \begin{definition} Let $p$ be a prime number. Consider the integers $f(x,y)$ for every $(x,y)$ in $\mathbb{Z}^2$. The $p$-adic\textit{ valuation tree} of $f$ is a rooted, labelled $p^2$-ary tree defined recursively as follows:
    
    Suppose that  $v_{0}$ be a root vertex at level $0$. There are $p^2$ edges from this root vertex to its $p^2$ children vertices at level $1$. These vertices correspond to all possible residue classes  $(i_{0},j_{0}) \mod p$.  Label the vertex corresponding to the class $(i_{0},j_{0}) $ with $0$ if $f(i_{0},j_{0}) \not\equiv 0 \mod p$ and with $*$ if $f(i_{0},j_{0}) \equiv 0 \mod p$.  If the label of a vertex is 0, it does not have any children. 
    
     If the label of a vertex is $*$, then it has $p^2$ children at level $2$. These vertices correspond to the residue classes $ (i_{0} + i_{1}p, j_{0} + j_{1}p)\mod p^2$ where $i_{1},j_{1} \in\{0,1,2,...,p-1\} $ and $(i_{0},j_{0}) \mod p$ is the class of the parent vertex.
     
     This process continues recursively so that at the $l^{\textrm{th}}$ level, there are $p^2$ children of any non-terminating vertex in the previous level $(l-1)$, each child of which corresponds to the residue classes  $(i_{0} + i_1p +...+i_{l-1}p^{l-1}, j_{0}+j_1p+ ...+ j_{l-1}p^{l-1})\mod p^l$.  Here $i_{l-1},j_{l-1}$ $\in \{0,1,2,...p-1\}$ and ($i_{0} + i_1p +...+i_{l-2}p^{l-2},  j_{0} + j_1p +...+j_{l-2}p^{l-2}) \mod p^{l-1}$ is the class of the parent vertex.  Label the vertex corresponding to the class $(i,j)$ with $l-1$ if $f(i,j) \not \equiv 0 \mod p^l$ and $*$ if $f(i,j)  \equiv 0 \mod p^l$.  Thus $ i = i_{0} + i_1p +...+i_{l-1}p^{l-1}, j = j_{0} + j_1p +...+j_{l-1}p^{l-1}$. \end{definition}
     \begin{example} Valuation tree of $x^2 + y^2 + xy +x + y +1$\end{example}
      \begin{tikzpicture}
      [sibling distance=4em,level distance=2em,
      every node/.style={shape=circle,draw= blue,align=right}]
        \node{}
        child{node{0}} 
        child{node{0}}
        child{node{0}}
       child{node{*}child{node{1}}
               child{node{1}}
               child{node{1}}
               child{node{1}}}        ;
        
      \end{tikzpicture}

      So \begin{center}
        \[
   v_2(x^2+y^2+xy+x+y+1)= 
\begin{cases}
   1,& \text{if } both \hspace*{.2cm} x, y \hspace*{.2cm} are \hspace*{.2cm}odd.\\
    0,              & \text{otherwise}
\end{cases}
\]
      \end{center}
       Hence the $2$-adic valuation of $x^2+y^2+xy+x+y+1$ admits a closed form. 
       \section{Some examples of valuation tree}
       We will achieve our goal of finding the $2$-adic valuation tree of general two degree polynomial by studying the $2$-adic valuation tree of some polynomials.
       \begin{example} The 2-adic valuation tree of $x^2 + y^2$ is as follows: \end{example}
       
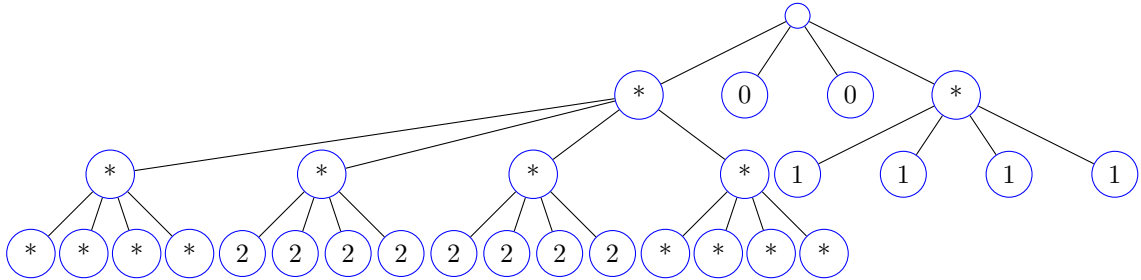
\begin{tikzpicture}
      [sibling distance=4em,level distance=3em,
      every node/.style={shape=circle,draw= blue,align=center}]
        \node{}
        child{node{*} [sibling distance =8em]
               child{node{*}[sibling distance =2em]
               child{node{*}}
               child{node{*}}
               child{node{*}}
               child{node{*}}}
               child{node{*}[sibling distance =2em]
               child{node{2}}
               child{node{2}}
               child{node{2}}
               child{node{2}}}
               child{node{*}[sibling distance =2em]
               child{node{2}}
               child{node{2}}
               child{node{2}}
               child{node{2}}}
               child{node{*}[sibling distance =2em]
               child{node{*}}
               child{node{*}}
               child{node{*}}
               child{node{*}}}
               child[missing]
               child[missing]}
       child{node{0}}
       child{node{0}}
       child{node{*} 
               child{node{1}}
               child{node{1}}
               child{node{1}}
               child{node{1}}}        ;
        
      \end{tikzpicture}\\
      
      By investigating the nature of above valuation tree,  we find an interesting pattern. Its study leads to a striking result.  We need some notation to state the result: Consider the binary representation of $b_k$ and $c_k$,\\ 
     \hspace*{2cm} $b_{k}$= $(i_{k}i_{k-1}...i_1i_0)_2$\\
      \hspace*{2cm} $c_{k}$= $(j_{k}j_{k-1}...j_1j_0)_2$ where $i_{0},i_{1},...i_{k-1}$, $j_{0},j_{1}...j_{k-1} $ $\in$ \{0,1\}.
      \begin{theorem}
       Let $v$ be a node at $k$-th level of the valuation tree of $x^2+y^2$.  Let the pair ($b_{k-1},c_{k-1})$ is associated to the vertex $v$. If we have $i_0 = i_1 =...= i_{k-2} = j_0 = j_1 =...= j_{k-2} = 0$, Then
      \begin{enumerate}
      \item If ($i_{k-1}, j_{k-1}$) = (0,0) then all four children of node $v$ will be labelled by $*$.  
      \item  If ($i_{k-1}, j_{k-1}$) = (1,1) then nodes descending from $v$ at $k+1, k+2,...,(2k-1)$-th levels will be labelled by $*$ and at $2k$-th level all nodes descending from $v$ will be labelled by $2k$-1.
    \item  If ($i_{k-1}, j_{k-1}$) =  (1,0) or (0,1) then all nodes descending from $v$ at $k+1, k+2,...,(2k-2)$-th levels will be labelled by $*$ and at ($2k$-1)-th level all nodes descending from $v$ will be labelled by $2k$-2.  
      \end{enumerate}
      \end{theorem}
       \begin{proof} We are given that \begin{center}
      $b_{k}$= $b_{k-1} +2^{k} i_{k}$ =($i_{k}i_{k-1}...i_1i_0)_2$\\ 
      $c_{k}$= $c_{k-1} +2^{k}j_{k} = (j_{k}j_{k-1}...j_1j_0)_2$ where\\ $i_{0},i_{1},...i_{k-1}$, $j_{0},j_{1}...j_{k-1} $ $\in$ \{0,1\} \end{center}
       When $(i_{k-1}, j_{k-1})$ = (0,0) then consider \begin{center}
      $b_{k}^2 + c_{k}^2$   mod $2^{k+1}$\\
      $\equiv$  $b_{k-1}^2 + c_{k-1}^2$  mod $2^{k+1}$,  k>0\\ but $b_{k-1} = c_{k-1} = 0$
      so all four children of node $v$ will be labelled by $*$.  
      \end{center}
     Let us consider
    \begin{center}
    $b_{k+l-1}^2 + c_{k+l-1}^2$ \hspace*{.25em} mod \hspace*{.25em}$2^{k+l}, 1\leq l \leq k$ \end{center}
     =$( 2^{k-1}i_{k-1} +...+2^{k+l-1}i_{k+l-1})^2+(2^{k-1}j_{k-1}+...+2^{k+l-1}j_{k+l-1})^2$ \hspace*{.25em} mod \hspace*{.25em}$2^{k+l}$\\ \begin{align}
      =   2^{2k-2}(i_{k-1}+2i_k+...+2^{l+1}i_{k+l-1})^2 + 2^{2k-2}(j_{k-1}+2j_k+...+2^{l+1}j_{k+l-1})^2\mod 2^{k+l}
     \end{align}
    \\
     If $(i_{k-1},j_{k-1})$ = (1,1) then the least power of 2 in (1) is $2^{2k-1}$. 
       so nodes descending from $v$ at $k+1, k+2,...,(2k-1)$-th levels will be labelled by $*$ and at $2k$-th level all nodes descending from $v$ will be labelled by $2k$-1.\\
       If $(i_{k-1},j_{k-1})$ = (1,0) or (0,1) then the least power of 2 in equation (1) is $2^{2k-2}$ so nodes descending from $v$ at $k+1, k+2,...,(2k-2)$-th levels will be labelled by $*$ and at $(2k-1)$-th level all nodes descending from $v$ will be labelled by $2k$-2.\end{proof}\begin{example}
       The $2$-adic valuation tree of $x^2+y^2+xy + x+y$ has also a specific pattern: \end{example}
        
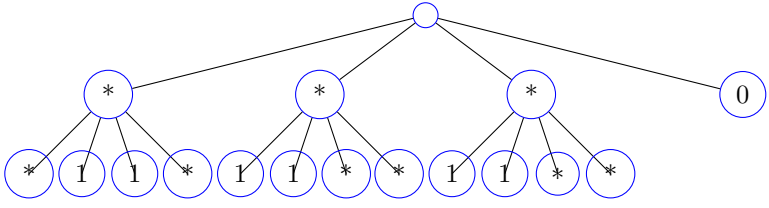
\begin{tikzpicture}
      [sibling distance=8em,level distance=3em,
      every node/.style={shape=circle,draw= blue,align=center}]
        \node{}
        child{node{*}[sibling distance=2em]
        child{{node{*}}}
        child{{node{1}}}
        child{{node{1}}}
        child{{node{*}}}}
        child{node{*}[sibling distance=2em]
        child{{node{1}}}
        child{{node{1}}}
        child{{node{*}}}
        child{{node{*}}}}
        child{node{*}[sibling distance=2em]
        child{{node{1}}}
        child{{node{1}}}
        child{{node{$*$}}}
        child{{node{*}}}}
        child{node{0}};
        \end{tikzpicture}\\
        We can formulate the following result for the 2-adic valuation tree of $x^2+y^2+xy + x+y$:
        \begin{theorem}
        Let $v$ be a node labelled with $*$ at level $k$ of the valuation tree of $x^2+y^2+xy + x+y$ for $k\geq 1$.  Then $v$ splits into four vertices at level $k+1$. Exactly two of them are labelled with $*$ and two are labelled with $k.$ The root vertex splits into three vertices with label $*$ and one vertex with label 0.
        \end{theorem}
        \begin{proof} Let the pair ($b_{k-1},c_{k-1}$) is associated to the vertex $v.$  So
   \begin{center}
   $b_{k} = 2^{k}i_{k} + b_{k-1}$\\
   $c_{k} = 2^{k}j_{k} + c_{k-1}$\\where $i_{0},i_{1},...i_{k}$, $j_{0},j_{1}...j_{k} $ $\in$ \{0,1\} and ($i_0,j_0) \neq$ (1,1)
   \end{center} We want to find ($i_k,j_k$) such that \begin{center} 
   $b_{k}^2 + c_{k}^2 + b_{k}c_{k} + b_{k} + c_{k} \equiv 0 \mod 2^{k+1}$
   \end{center}
   On putting the expression of $b_{k}$ and $c_{k}$ in above equation, we got \begin{center}
   $b_{k-1}^2 + c_{k-1}^2 + b_{k-1}c_{k-1}+ b_{k-1} + c_{k-1} + 2^k(b_{k-1}j_k+ c_{k-1}i_k + i_{k-1} + j_{k-1}) \equiv 0 \mod 2^{k+1}$
   \end{center} We know that $b_{k-1}^2 + c_{k-1}^2 + b_{k-1}c_{k-1}+ b_{k-1} + c_{k-1}$ = $a2^k$,  $a\in \{0,1\}$  so above equation becomes \begin{center}
   $a2^k + 2^k(b_{k-1}j_k+ c_{k-1}i_k + i_k + j_k)$ $\equiv$ 0 mod $2^{k+1}$\\
   = $a +b_{k-1}j_k+ c_{k-1}i_k + i_k + j_k$ $\equiv$ 0 $\mod$ 2
  \begin{align}
   =  i_k(j_0+1) + j_k(i_0+1) \equiv \hspace*{.3em} a \hspace*{.3em} \mod \hspace*{.3em} 2 
   \end{align}
   \end{center}
   Now if ($i_0,j_0)$ = (0,0) then (2) becomes $i_k + j_k$ $\equiv$ $a$ mod 2.  Hence there are two vertices labelled with $k$ descending from $v$ with $i_k + j_k$ $\not\equiv$ $a$ mod 2 and other two vertices are not terminating labelled with $*$.\\
   If ($i_0,j_0)$ = (1,0) then equation(2) becomes $i_k $ $\equiv$ $a$ mod 2.  Hence there are two vertices labelled with $k$ descending from $v$ with $i_k$ $\not\equiv$ $a$ mod 2 and other two vertices are not terminating labelled with $*$. Similarly for ($i_0,j_0)$ = (0,1), equation(2) becomes $j_k $ $\equiv$ $a$ mod 2.  Hence there are two vertices labelled with $k$ descending from $v$ with $j_k$ $\not\equiv$ $a$ mod 2 and other two vertices are not terminating labelled with $*$.\end{proof}
   \begin{example} The $2$-adic valuation tree of $xy + x+y+1$ is as follows: \end{example}
    \begin{tikzpicture}
      [sibling distance=4.96em,level distance=3em,
      every node/.style={shape=circle,draw= blue,align=center}]
      \node{}
      child{node{0}}
      child{node{*}[sibling distance=4em]
               child{node{1}}
               child{node{*}[sibling distance = 2em]
                  child{node{2}}
                  child{node{*}}
                  child{node{2}}
                  child{node{*}}}
                  child{node{1}}
                  child[missing]
                child{node{*}[sibling distance = 2em]
                  child{node{2}}
                  child{node{*}}
                  child{node{2}}
                  child{node{*}}}}
      child{node{*}}
      child[missing]
      child[missing]
      child[missing]
      child{node{*}
      [sibling distance=2em]
               child{node{*}[sibling distance = 2em]
                  child{node{2}}
                  child{node{2}}
                  child{node{2}}
                  child{node{2}}}child[missing]child[missing]child[missing]
               child{node{*}
               child{node{*}}
               child{node{*}}
               child{node{*}}
               child{node{*}}}child[missing]
                child{node{*}}child[missing]child[missing]
                child{node{*}
                child{node{*}}child{node{*}}child{node{*}}
                child{node{*}}}
                }
      
      ;
      \end{tikzpicture}\\
     By analysing the pattern in the above tree, we can state following result:
       \begin{theorem}Let $v$ be a node labelled with * at $k$-th level of the valuation tree of $xy + x+ y+1$, for $k$>0.  Then this vertex $v$ splits into four nodes such that
      \begin{enumerate}
      \item If $(i_0,j_0)$ = (1,1) then all four nodes will be labelled by $*$ or k+1.
      \item If $(i_0,j_0)$=  (0,1), (1,0) and (0,0) then exactly two will be labelled by $*$.
       \end{enumerate}The root vertex $v_0$ splits into four vertices, each labelled by $*$.
      \end{theorem}
         \begin{proof} Let $(b_{k-1},c_{k-1})$ is associated to the vertex $v$ at $k$-th level of the valuation tree of $xy + x+ y+1$, for $k$>0.  So  \begin{center}
   $b_{k} = 2^{k}i_{k} + b_{k-1}$\\
   $c_{k} = 2^{k}j_{k} + c_{k-1}$\\where $i_{0},i_{1},...i_{k}$, $j_{0},j_{1}...j_{k} $ $\in$ \{0,1\}.
   \end{center} We want to find ($i_k,j_k$) such that \begin{center} 
   $b_{k}c_{k} + b_{k} + c_{k} + 1$$ \equiv$ 0 $\mod$ $2^{k+1}$
   \end{center}
   On putting the expression of $b_{k}$ and $c_{k}$ in above equation, we got \begin{center}
   $ b_{k-1}c_{k-1}+ b_{k-1} + c_{k-1} +1+ 2^k(b_{k-1}j_k+ c_{k-1}i_k + i_{k-1} + j_{k-1})$ $\equiv$ 0 $\mod$ $2^{k+1}$
   \end{center} We know that $b_{k-1}c_{k-1}+ b_{k-1} + c_{k-1}+1$ = $a2^k$,  $a\in \{0,1\}$  so above equation becomes \begin{center}
   $a2^k + 2^k(b_{k-1}j_k+ c_{k-1}i_k + i_k + j_k)$ $\equiv$ 0 $\mod$ $2^{k+1}$\\
   = $a +b_{k-1}j_k+ c_{k-1}i_k + i_k + j_k$ $\equiv$ 0 $\mod$ 2
  \begin{align}
   =  i_k(j_0+1) + j_k(i_0+1) \equiv \hspace*{.3em} a \hspace*{.3em}\mod\hspace*{.3em} 2 
   \end{align}
   \end{center} Now if ($i_0,j_0)$ = (0,0) then equation(3) becomes $i_k + j_k$ $\equiv$ $-a$ $\mod$ 2. Hence there are two vertices labelled with $k$ descending from $v$ with $i_k + j_k$ $\not\equiv$ $-a$ $\mod$ 2 and other two vertices are not terminating labelled with $*$.\\
   If ($i_0,j_0)$ = (1,0) then equation(3) becomes $i_k $ $\equiv$ $a$ $\mod$ 2.  Hence there are two vertices labelled with $k$ descending from $v$ with $i_k$ $\not\equiv$ $a$ $\mod$ 2 and other two vertices are not terminating labelled with $*$. Similarly for ($i_0,j_0)$ = (0,1), equation(3) becomes $j_k $ $\equiv$ $a$ $\mod$ 2.  Hence there are two vertices labelled with $k$ descending from $v$ with $j_k$ $\not\equiv$ $a$ $\mod$ 2 and other two vertices are not terminating labelled with $*$. 
   \\If $(i_0,j_0)$ = (1,1) then equation(3) becomes $0 \equiv a \mod 2$. Hence $v$ splits into four nodes labelled by $*$ or $k$ depending upon whether $a \equiv 0 \mod 2$ or $a \equiv 1 \mod 2$. \end{proof}
   \section{The algebraic meaning of an infinite branch of the valuation tree}
   In the last section, we have seen examples with infinite branches.  But what is the algebraic meaning of a such a phenomena? From the definition of valuation tree we can deduce the following results:\begin{lemma} Let $i_0,j_0 \in \{0,1,2,...,p-1\}$ such that $f(i_0,j_0)$ $\not \equiv 0 \mod p$ then \begin{center}
   $v_p(f(i,j))$ = 0 for $i \equiv i_0 \mod p$ and $j \equiv j_0 \mod p$.
\end{center}    
   \end{lemma}
   \begin{lemma} Assume $i_0,j_0, i_1, j_1 \in \{0,1,2,...,p-1\}$ such that $f(i_0,j_0)$ $ \equiv 0 \mod p$ and $f(i_0 + i_1p, j_0+j_1p)$ $\not \equiv 0 \mod p^2$ then \begin{center}
   $v_p(f(i,j))$ = 1 for $i \equiv i_0+i_1p\mod p^2$ and $j \equiv j_0+j_1p\mod p^2$.
   \end{center} \end{lemma}Continuing this process produces the next lemma: \begin{lemma} Let $i_0,j_0, i_1, j_1,...,i_{n},j_{n} \in \{0,1,2,...,p-1\}$ satisfying \begin{center}
   $f(i_0,j_0)$ $\equiv 0 \mod p$\\
   $f(i_0+i_1p,j_0+j_1p)$ $\equiv 0 \mod p^2$\\
   $f(i_0+i_1p,j_0+j_1p)$ $\equiv 0 \mod p^2$\\
   ...  ...  ...\\$f(i_0+i_1p+...+i_{n-1}p^{n-1},j_0+j_1p+...+j_{n-1}p^{n-1})$ $\equiv 0 \mod p^{n+1}$ and \\
   $f(i_0+i_1p+...+i_np^n,j_0+j_1p+...+j_np^n)$ $\not\equiv 0 \mod p^{n+1}$.
   \end{center}Then any $(i,j)$ = $(i_0+i_1p+...+i_np^n,j_0+j_1p+...+j_np^n)$ $\mod p^{n+1}$, satisfies \begin{center}
   $v_p(f(i,j))$ = $n$.
   \end{center}
    
   \end{lemma}
   \begin{theorem}
   Any infinite branch in the tree  associated to polynomial $f(x,y)$ corresponds to a root of $f(x,y$) = 0 in $\mathbb{Q}_p^2$.
   \end{theorem}
   \begin{proof} Let the sequence of indices generated to come an infinite branch of the tree at $n$-th level is $(a_n,b_n)$ = $(i_0 + i_1p +...+ i_{n-1}p^{n-1},  j_0 + j_1p +...+ j_{n-1}p^{n-1}$) where $i_{0},i_{1},...i_{n-1}$ and  $j_{0},j_{1}...j_{n-1} $ $\in \{0,1,2,...p-1\}$ such that\\
     \hspace*{4cm}  $f(i_{0}$ ,$j_{0}$)$\equiv 0 \mod p$\\
     \hspace*{3cm}    $f(i_{0} + i_1p, j_{0}+j_1p)$ $\equiv 0 \mod$ $p^2$\\
     \hspace*{2cm}   $f(i_{0} + i_1p + i_{2}p^2 , j_{0}+j_1p+ j_{2}p^2)$ $\equiv   0 \mod$ $p^3$\\
     \hspace*{5cm} ..........\\\\
     $f(i_{0} + i_1p +...+i_{n-1}p^{n-1}, j_{0}+j_1p+ ...+ j_{n-1}p^{n-1})$ $\equiv 0\mod$ $p^n$ and\\ $f(i_{0} + i_1p + i_{2}p^2 +...+i_{n}p^n, j_{0}+j_1p+ j_{2}p^2 + ...+ j_{n}p^n)$ $\not$$\equiv$  $ 0\mod$ $p^{n+1}$ \\
      Now $a_n$ and $b_n$ satisfy: 
      0 $\leq$  $a_n ,  b_n$ $\leq$ $ p^n$ and $a_n$ $ \equiv$ $ a_{n+1}$ $ mod$ $ p^n$ , $ b_n$ $\equiv$ $ b_{n+1}$  $mod$ $ p^n$.  Hence sequences $a_n$ and $b_n$ are convergent to some element in the field $\mathbb{Q}_p$.  Let ($a_n ,  b_n$) converges to $(x,y)$ for $x , y$ $\in$  $\mathbb{Q}_p$. Since the polynomial $f(x,y)$ is continuous so $f(a_n,b_n)$ converges to $f(x,y)$. Now by lemma (3.3), we know that $v_p(f(a_n,b_n))$ tends to $\infty$ as $n$ tends to $\infty$ so $f(a_n,b_n)$ tends to 0 when $n$ tends to $\infty$. Hence $f(x,y)$ = 0. \end{proof} 
      \begin{corollary} The $p$-adic valuation $v_p(f(x,y))$ admits a closed form formula (there exist a natural  number $M$ such that $v_p(f(x,y))<M$ $ \forall$ $ x,y$) if the equation $f(x,y)$ = 0 has no solution in $\mathbb{Q}_p^2.$\end{corollary}\par Since the polynomials in Examples (3.1-3.3) have zeros in $\mathbb{Q}_2^2$ so they do not admit closed form formula for $2$-adic valuation.
      \section{The 2-adic valuation tree of \textbf{$ax^2+by^2$}} 
      We are slowly inching towards our goal of findind the $2$-adic valuation tree of the general two degree polynomial $f(X,Y)\in \mathbb{Z}[X,Y]$.  We can generalise Example (3.1).  Consider the binary representation of $b_k$ and $c_k$: \begin{center}$b_{k}$= $(i_{k}i_{k-1}...i_1i_0)_2$
      , $c_{k}$= $(j_{k}j_{k-1}...j_1j_0)_2$ \end{center} where $i_{0},i_{1},...i_{k}$, $j_{0},j_{1}...j_{k} $ $\in$ \{0,1\}. \begin{theorem}
        Let $v$ be a vertex at $k$-th level of the valuation tree of $ax^2+by^2$ where $a = 2^n\alpha$, $b = 2^m\beta$,  $\alpha$  and $\beta$ are odd. Let $\gamma = min(m,n)$.  Suppose that the pair ($b_{k-1},c_{k-1})$,  in the above notation is associated to vertex $v$.  Further suppose that $i_0 = i_1 =...= i_{k-2} = j_0 = j_1 =...= j_{k-2} = 0$.  Then
      \begin{enumerate}
      \item the pair ($i_{k-1}, j_{k-1}$) = (0,0) implies  all four children of vertex $v$ are labelled by $*$.  
      \item  If ($i_{k-1}, j_{k-1}$) = (1,1) and if $w$ is a vertex descending from $v$ at $l$-th level then $w$ is labelled by $*$ whenever $l$ $\in$ \{$k+1, k+2,...,(2k+\gamma-1)$\}  and by $(2k+ \gamma$-1) for $l$ = $2k + \gamma$. 
    \item  If ($i_{k-1}, j_{k-1}$) =  (1,0) or (0,1) and if $w$ is a vertex descending from $v$ at $l$-th levels then $w$ is labelled by $*$ whenever $l$ $\in$ \{$k+1, k+2,...,...,(2k+\gamma-2)$\} and by $2k+ \gamma-2$ when $l$ = $2k+ \gamma-1$.
      \end{enumerate}
\end{theorem} 
 Proof: We are given that \begin{center}
      $b_{k}$= $b_{k-1} +2^{k} i_{k}$ =($i_{k}i_{k-1}...i_1i_0)_2$\\ 
      $c_{k}$= $c_{k-1} +2^{k}j_{k} = (j_{k}j_{k-1}...j_1j_0)_2$ where\\ $i_{0},i_{1},...i_{k-1}$, $j_{0},j_{1}...j_{k-1} $ $\in$ \{0,1\}. \end{center}
       When $(i_{k-1}, j_{k-1})$ = (0,0) then consider \begin{center}
      $ab_{k}^2 + bc_{k}^2$  $ mod$ $2^{k+1}$
      $\equiv$  $ab_{k-1}^2 + bc_{k-1}^2$ $\mod$ $2^{k+1}$,  $k>0$\\ But $b_{k-1} = c_{k-1} = 0$
      and so,  all four children of node $v$ will be labelled by $*$.  
      \end{center}
     Let us consider
    \begin{center}
    $ab_{k+l-1}^2 + bc_{k+l-1}^2 \hspace*{.25em} mod \hspace*{.25em}2^{k+l}, 1\leq l \leq k$ \end{center}
     =$a( 2^{k-1}i_{k-1} +...+2^{k+l-1}i_{k+l-1})^2+b(2^{k-1}j_{k-1}+...+2^{k+l-1}j_{k+l-1})^2 \hspace*{.25em} mod \hspace*{.25em}2^{k+l}$\\ \begin{align}
      =   2^{2k+\gamma-2}(\alpha(i_{k-1}+2i_k+...+2^{l+1}i_{k+l-1})^2 + \beta(j_{k-1}+2j_k+...+2^{l+1}j_{k+l-1})^2) \hspace*{.5em}\mod \hspace*{.5em} 2^{k+l}
     \end{align}
    \\
     If $(i_{k-1},j_{k-1})$ = (1,1) then the least power of 2 in equation (4) is $2^{2k+\gamma-1}$. 
       so nodes descending from $v$ at $k+1, k+2,...,(2k+\gamma-1)$-th levels will be labelled by $*$ and at $(2k+\gamma$)-th level all nodes descending from $v$ will be labelled by $2k+\gamma-1$.\\
       If $(i_{k-1},j_{k-1})$ = (1,0) or (0,1) then the least power of 2 in equation (4) is $2^{2k+\gamma-2}$ so nodes descending from $v$ at $k+1, k+2,...,(2k+\gamma-2)$-th levels will be labelled by $*$ and at $(2k+\gamma-1)$-th level all nodes descending from $v$ will be labelled by $2k+\gamma-2$.      
   \section{The 2-adic valuation tree of the general polynomial \textbf{$ax^2+by^2+cxy+dx+ey+g$}}
   Let $f(x,y) = ax^2+by^2+cxy+dx+ey+g$,  for $a,b,c,d,g \mathbb{Z}.$The following theorem describes the valuation tree of $f(x,y)$:
   \begin{theorem}
   Let $v$ be a vertex at $k$-th level of the valuation tree of $f(x,y)$ labelled by $*$ for $k>1$. Then $v$ splits into four vertices such that either all are non-terminating or two of them are non-terminating.  For $k$=1,  the labelling of vertices depends upon the coefficients of $f(x,y)$. 
   \end{theorem}
   \begin{proof} We are given that $f(x,y)$ =  $ax^2 + by^2 +cxy + dx + ey$. Let $(b_{k-1},c_{k-1})$ is associated to the vertex $v$ at $k$-th level of the valuation tree.  so we have \begin{center}
   $f(b_{k-1},c_{k-1}) \equiv 0 \mod 2^k$, where
   \end{center} \begin{center}
    $b_{k}$= $(i_{k}i_{k-1}...i_1i_0)_2$.
      $c_{k}$= $(j_{k}j_{k-1}...j_1j_0)_2$, \end{center}Here $i_{0},i_{1},...i_{k}$, $j_{0},j_{1}...j_{k} $ $\in$ \{0,1\} and $(b_0,c_0) = (i_0,j_0)$.\\We want to find $(i_k,j_k)$ such that \begin{center}
      $f(b_k,c_k) \equiv 0 \mod 2^{k+1}$.
      \end{center} On putting the expression for $(b_k,c_k)$ in the above equation we get \begin{center}
      $a(b_{k-1} + 2^{k}i_{k})^2) + b(c_{k-1}+2^{k}j_{k})^2 + c(b_{k-1}+2^{k}i_{k})(c_{k-1}+2^{k}j_{k}) + d(b_{k-1}+2^{k}i_{k}) + e(c_{k-1}+2^{k}j_{k})+g \equiv 0 \mod 2^{k+1}, k>0$.\end{center}But we know that $f(b_{k-1},c_{k-1})$ $\equiv$ 0 $mod$ $2^{k}$ so  $f(b_{k-1},c{k-1})$ = $\alpha 2^{k}, \alpha \in \{0,1\}$.
      Hence we want to find $(i_k,j_k)$ such that \begin{center}$\alpha2^k+ 2^k[i_k(cc_{k-1}+d) + j_k(cb_{k-1}+e)] \equiv 0 \mod 2^{k+1}$,\end{center}That is\begin{center} $\alpha+ i_k(cj_0+d)+j_k(ci_0+e) \equiv 0 \mod 2$
      \end{center}The following Table 5.1 gives the all possible cases for $(i_k,j_k)$ when $\alpha\equiv 0 \mod 2$.

       \begin{longtable}{|c|c|c|c|c|c|c|}
      
      \hline
      Serial no.&($i_{k},j_{k}$) & $c$&$d$&$e$&$(i_0,j_0)$& Label \\
      \hline
      1&(0,0)& -&-&-&-&$*$\\ \hline
      2&(1,0)&odd&odd&-&(0,0),(1,0)&k\\ \hline
      3&(1,0)&odd&odd&-&(0,1),(1,1)&$*$\\ \hline
      4&(1,0)&odd&even&-&(0,0),(1,0)&$*$\\ \hline
      5&(1,0)&odd&even&-&(1,1),(0,1)&k\\ \hline
      6&(1,0)&even&odd&-&-&k\\ \hline
      7&(1,0)&even&even&-&-&$*$\\ \hline
      8&(0,1)&odd&-&odd&(0,0),(0,1)&k\\ \hline
      9&(0,1)&odd&-&odd&(1,0),(1,1)&$*$\\ \hline
      10&(0,1)&odd&-&even&(0,0),(1,0)&$*$\\ \hline
      11&(0,1)&odd&-&even&(1,1),(0,1) &k \\ \hline
      12&(0,1)&even&-&odd&-&k\\ \hline
      13&(0,1)&even&-&even&-&$*$\\ \hline
      14&(1,1)&odd&odd&even&(0,0),(1,1)&k\\ \hline
      15&(1,1)&odd&odd&even&(1,0),(0,1)&$*$\\ \hline
      16&(1,1)&odd&odd&odd&(0,0),(1,1)&$*$\\ \hline
      17&(1,1)&odd&odd&odd&(1,0),(0,1)&k\\ \hline
      18&(1,1)&odd&even&even&(0,0),(1,1)&$*$\\ \hline
      19&(1,1)&odd&even&even&(1,0),(0,1)&k\\ \hline
      20&(1,1)&odd&even&odd&(0,0),(1,1)&k\\ \hline
      21&(1,1)&odd&even&odd&(1,0),(0,1)&$*$\\ \hline
      22&(1,1)&even&odd&even&-&k\\ \hline
      23&(1,1)&even&odd&odd&-&$*$\\ \hline
      24&(1,1)&even&even&even&-&$*$\\ \hline
      25&(1,1)&even&even&odd&-&k\\ \hline
      
      \end{longtable} \begin{center}
      $ \textbf{Table 5.1}$
       \end{center}  \par Hence the theorem is proved for $\alpha\equiv 0 \mod2$.
      For the case $\alpha\equiv 1 \mod2$, we can find the appropriate label by interchanging the $*$ and $k$ in the last column of Table 5.1.  Hence the theorem is proved in this case as well. \end{proof}
      \section{One step towards the general polynomial}
      One step in the direction to find the valuation tree of general 2 variable polynomial with coefficients in $ \mathbb{Z}$ is the following theorem about the $2$-adic valuation tree of degree 3 polynomial $x^2y+5$.
      \begin{theorem}
       Let $v$ be a node labelled with $*$ at level $k$ of the valuation tree of $x^2y+5$ for $k\geq 1$.  Then $v$ splits into four vertices at level $k+1$. Exactly two of them are labelled with $*$ and two are labelled with $k.$ The root vertex splits into three vertices with label 0 and one vertex with label $*$.
\end{theorem}
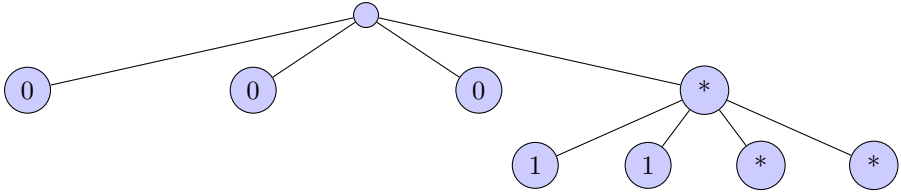
\begin{tikzpicture}[nodes={draw, circle,fill=blue!20,}, -,level distance = 1cm]
 
   \node{}
    child { node {0} }
    child[missing]
    child { node{0}  }
    child[missing]
    child { node{0}  }
    child[missing]
    child { node{*}
    child{node{1}}
    child{node{1}}
    child{node{*}}
    child{node{*}}}  ;

\end{tikzpicture}\\

  \begin{proof}
  We are given that $f(x,y)$ =  $x^2y+5$. Let $(b_{k-1},c_{k-1})$ is associated to the vertex $v$ at $k$-th level of the valuation tree.  so we have \begin{center}
   $f(b_{k-1},c_{k-1}) \equiv 0 \mod 2^k$, where
   \end{center} \begin{center}
    $b_{k}$= $(i_{k}i_{k-1}...i_1i_0)_2$.
      $c_{k}$= $(j_{k}j_{k-1}...j_1j_0)_2$, \end{center}Here $i_{0},i_{1},...i_{k}$, $j_{0},j_{1}...j_{k} $ $\in$ \{0,1\} and $(b_0,c_0) = (i_0,j_0)=(1,1)$.\\We want to find $(i_k,j_k)$ such that \begin{center}
      $f(b_k,c_k) \equiv 0 \mod 2^{k+1}$.
      \end{center} On putting the expression for $(b_k,c_k)$ in the above equation we get 
    \begin{align}
     b_{k-1}^2c_{k-1}+2^kj_kb_{k-1}^2+5 \equiv 0 \mod 2^{k+1}.  
    \end{align}
    We are given that $f(b_{k-2},c_{k-2})$ $\equiv 0 \mod 2^{k-1}$ so \begin{center}
     $b_{k-2}^2c_{k-2} + 5$ = $2^{k-1}a$ ,    
\end{center}     where $a \in \{0,1\}$.
     Hence (5) becomes \begin{align} 
     2^{k}(b_{k-1}^2j_{k} + a ) \equiv 0 \mod 2^{k+1} 
     \end{align}
  Now observe that $b_{k-1}$ $\equiv  1 \mod 2$.  Hence (6) becomes \begin{center}
    $2^{k}(j_{k} + a )$ $\equiv 0 \mod 2^{k+1}$ or \\
     $j_{k} $ $\equiv a \mod 2$
  \end{center}
   Therefore there are two vertices labelled with $k$ descending from $v$ with $j_{k} $ $\not\equiv a \mod 2$  and other two vertices are non terminating labelled with $*$. \end{proof} 
   \par We can prove the theorem 6.1 by using the generalized  Hensel Lemma \cite{seven}. 
   \begin{proof}
   Let $f(x,y) = (x^2y + 5,  x + 1)$ and $ a = (1, 1) $, so\begin{center}
     $f(1,1) = (6,6)$ and $J_f(x,y) = \begin{array}{|cc|} 
                                       2xy & x^2\\
                                       1&0
                                       \end{array}$ = $-x^2$\\
   $||f(1,1)||_2 = \dfrac{1}{2} ,   |J_f(1, 1)|_2 = 1  $. Also $ ||f(1,1)||_2^2  <  |J_f(1, 1)|_2$
\end{center}                                  So by generalized Hensel Lemma \cite{seven} there is a unique solution to $f(x,y) = (0,0)$ in $\mathbb{Z}_2^2$ such that $||(x,y) - (1,1) ||_2$ $<$ 1.  The vector $\left(\begin{matrix}
 x\\y
 \end{matrix}\right)$ is the limit of sequence $\alpha_n$ =   $\left(\begin{matrix}
 x_n\\y_n \end{matrix}\right)$ where $\alpha_1$ =  $a$ =   $\left(\begin{matrix}
 1\\1\end{matrix}\right)$ and for n $\geq$ 1\begin{align}
  \alpha_{n+1} = \left(\begin{matrix} 
   x_n\\y_n \end{matrix}\right) - \begin{bmatrix}
   2x_ny_n & x_n^2\\
   1& 0
   \end{bmatrix}^{-1}\begin{bmatrix}
   x_n^2y_n + 5\\
   x_n + 1
   \end{bmatrix} \end{align}Let the pair ($b_{k-1},c_{k-1})$ is associated to vertex $v$. So\begin{center} $f(b_{k-1},c_{k-1}) \equiv 0 \mod 2^k $,\end{center}Also from above equation we will get $(b_{k-1}, c_{k-1})$ = (-1, -5) $\mod$ $2^k$\\
       Now using the above expression in (7) we get\\ \hspace*{2cm} $\alpha_{k+1} = (b_{k}, c_{k})$ = (-1, -5) $\mod$ $ 2^{k+1}$\\
       Hence $f(b_{k}, c_{k}) \equiv$ 0 $\mod$  $2^{k+1}$.\\ Similarly, on replacing x+1 by y+5 in $f(x,y)$ we will get\\\hspace*{2cm} $(b_{k}, c_{k})$ = (1, -5) $\mod$ $2^{k+1}$ and\\
       \hspace*{2cm} $f(b_{k}, c_{k}) \equiv$ 0 $\mod$  $2^{k+1}$\\
       Hence by definition we found two nodes of the valuation tree of $x^2y + 5$ labelled by $*$.   \end{proof}

\end{document}